\documentclass[runningheads]{llncs}
\usepackage{graphicx}
\usepackage{amsfonts,amssymb}
\usepackage{amsmath}
\usepackage{mathrsfs}
\usepackage{graphicx}
\usepackage{indentfirst}
\usepackage[colorlinks=true, allcolors=blue]{hyperref}
\usepackage{graphicx}
\usepackage{float}
\usepackage{booktabs}
\usepackage{color}

\newcommand{\ud}{\mathrm{d}}
\newcommand{\bd}{\boldsymbol}

\newcommand{\Ito}{It\^{o}}
\newcommand{\R}{\mathbb R}
\newcommand{\grad}{\mathrm{grad}}

\begin{document}
\title{The Riemannian Langevin equation and conic programs}
\titlerunning{Riemannian Langevin equation}
\author{Govind Menon, Tianmin Yu}

\authorrunning{G.Menon and T. Yu}
\institute{Division of Applied Mathematics, \\ Brown University, Providence RI 02912, USA\\ 
\email{govind\_menon@brown.edu, tianmin\_yu@brown.edu}}
%
\maketitle              
\begin{abstract}
Diffusion limits provide a framework for the asymptotic analysis of stochastic gradient descent (SGD) schemes used in machine learning. We consider an alternative framework, the Riemannian Langevin equation (RLE), that generalizes the classical paradigm of equilibration in $\mathbb{R}^n$ to a Riemannian manifold $(\mathcal{M}^n,g)$. The most subtle part of this equation is the description of Brownian motion on $(\mathcal{M}^n,g)$. Explicit formulas are presented for some fundamental cones. 

\keywords{Stochastic gradient descent  \and Riemannian Langevin equation.}
\end{abstract}

\section{Introduction}
\label{sec:intro}
\subsection{Stochastic gradient descent}
Stochastic gradient descent (SGD) schemes in machine learning typically arise as follows. An empirical loss function $E: \mathbb{R}^n \to \mathbb{R}$, for a training parameter $x$, is defined through a finite sum
$E(x) = \frac{1}{N}\sum_{i=1}^N \varepsilon_i(x)$, where $\varepsilon_i(x)=\varepsilon(x;y_i,z_i)$ denotes a loss function evaluated on a finite set of training data $\{(y_i,z_i)\}_{i=1}^N$. The loss function is minimized using the stochastic gradient descent scheme
\begin{equation}
\label{eq:sgd2}
x^{k+1} = x^k  -\gamma_k \nabla \varepsilon_{i_k} (x^k),
\end{equation}
where $i_k$ is chosen randomly from the set $\{1,\ldots, N\}$ and $\gamma_k$ is a time step. 

Several variants of SGD have been explored since the classic work of Robbins and Monro~\cite{Robbins}. What is different in modern maching learning is the large size of $n$ and $N$ and the protocol for the learning rate $\gamma$. Diffusion limits of SGD replace the discrete iteration above with stochastic differential equations (SDE); these SDE depend on the manner in which $\gamma_k\to 0$, $n \to \infty$ and $N \to \infty$. Some examples of this approach are the stochastic modified equation (SME) proposed in~\cite{Li-2019}, the variational analysis using Kullback-Leibler divergence proposed in~\cite{Mandt}, and {\em homogenized\/} SGD (HSGD) defined in~\cite{Paquette-SGD}. 

\subsection{Riemannian Langevin equation}
SDE limits of SGD schemes begin with an algorithm and study its scaling limits. The approach in this paper is different. We begin with diffusions that extend the classical Langevin equation to a Riemannian setting.
The relation to optimization lies in the nature of the underlying Riemannian geometry. Let us first explain the model; we then explain why it is a natural extension of ideas used in classical and modern optimization theory. 

Recall that the Langevin equation associated to the loss function $E: \mathbb{R}^n\to \mathbb{R}$, at inverse temperature $\beta>0$, is formulated mathematically as the \Ito\/ SDE
\begin{equation}
 \label{eq:langevin1}
 dX_t = -\nabla E(X_t) \, dt + \sqrt{\frac{2}{\beta}} dB_t,
 \end{equation}
where $\{B_t\}_{t \geq 0}$ denotes standard Brownian motion on $\R^n$. Given an $n$-dimensional Riemannian manifold $(\mathcal{M},g)$ with metric $g$ and a loss function $E: \mathcal{M} \to \R$ we consider the {\em Riemannian Langevin equation\/} (RLE)
\begin{equation}
 \label{eq:rl1}
 dX_t = -\grad\, E(X_t) \, dt + dB^{g,\beta}_t.
\end{equation}

Both the gradient and the Brownian motion at inverse temperature $\beta$ are now computed with respect to the Riemannian metric $g$. In particular, the Brownian motion $B_t^{g,\beta}$ on $(\mathcal{M},g)$ must be defined carefully as discussed below. Let $f \in C^\infty(\mathcal{M})$ be a test function. The infinitesimal generator $L$ of the diffusion~\eqref{eq:rl1} is
\begin{align}
    \mathcal Lf=-\mathrm{grad}\, E(f)+\frac1\beta\Delta f, \quad\mathrm{where}\quad\Delta f=\frac{1}{\sqrt {\det g}}\partial_i (\sqrt{\det g} g^{ij}\partial_j f),
\end{align}
is the Laplace-Beltrami operator, $\mathrm{grad}$ denotes the gradient with respect to $g$, and the volume form is computed in coordinates using $\det g=\det (g_{ij})$.

The Fokker-Planck equation, $\partial_t \rho =L^*\rho$, where the dual is with respect to the volume form of $g$, takes the form 
\begin{align}
    \partial_t \rho=\mathrm{div}\,\left( \rho\,\mathrm{grad }F\right), \quad F= E+\frac1\beta\log\rho.
\end{align}
The free energy, $F$, is constant in equilibrium and we find the Gibbs density
\begin{equation}
\label{eq:rl2-gibbs}
    \rho(x) = \frac{1}{Z_\beta} e^{-\beta E(x)}, \quad Z_\beta = \int_{\mathcal{M}} e^{-\beta E(x)} \sqrt{\det{g}} \, dx.
\end{equation}

RLE is a method a method to study the Gibbs measure associated to $F$, whereas SGD schemes seek the minimum of $F$. However, these techniques are closely related. When $\beta \to \infty$ and $F$ has a unique global minimum at $x_* \in \mathcal{M}$, the Gibbs measure concentrates at $x_*$ as $\beta \to \infty$ with rigorous asymptotics provided by large deviations theory. A subtle feature of the metrics arising in optimization is that the volume $\int_M \sqrt{\det g}\, dx$ may be infinite.

\subsection{Riemannian geometries in optimization}
The framework of RLE provides a natural geometric unity between conic programs and deep learning. What changes is the underlying Riemannian manifold $(\mathcal{M},g)$. Let us explain this idea through examples. 

Bayer and Lagarias systematized the Riemannian geometry discovered by Karmarkar for interior-point methods~\cite{Bayer,Karmarkar}. We focus on the {\em canonical barrier\/}~\cite{Hildebrand}. Associated to every regular convex cone $K \subset \mathbb{R}^n$ is a unique convex function $F$ defined in the interior $K^o$ of $K$ such that $F(x) \to +\infty$ as $x \to \partial K$. This function is the Cheng-Yau solution to the Monge-Amp\`ere equation
\begin{align}
\label{eq:monge-ampere}
    F=\frac12\log\,\det\, D^2F, \quad x \in K^o.
\end{align}
Given a barrier and a vector $c \in \mathbb{R}^n$, the conic program $\min_{x \in C} c^Tx$ is solved by taking the $\theta \to \infty$ limit of the {\em central path\/}
\begin{equation}
\label{eq:cp1} x(\theta) =\mathrm{argmin}_{s \in C} \{F(s) + \theta c^Ts \}.
\end{equation}
Further, $x(\theta)$ above is the solution to the Riemannian gradient flow
\begin{equation}
\label{eq:cp2} \frac{dx}{d\theta} = -\mathrm{grad}_g (c^Tx), \quad g= {D^2F}, \quad x(0)= \mathrm{argmin} F.  
\end{equation}
That is, the Hessian of the barrier $F$ provides the underlying Riemannian metric. The canonical barrier has several striking geometric properties~\cite{Hildebrand}.

Riemannian metrics have also been extensively used in geometric deep learning~\cite{Bronstein}. A model problem that allows a comparison between deep learning and classical optimization is the deep linear network~\cite{Arora-Cohen-Hazan,Bah,Cohen-Menon-Veraszto}. The training space for a network of depth $N$ is the product space of $d\times d$ matrices $\mathbb{M}_d^N$. Given $\mathbf{W}=(W_N,W_{N-1}, \ldots, W_1)$ the observable is the product $V=W_{N}W_{N-1}\cdots W_1$. Learning problems like matrix completion may be modeled as a Euclidean gradient descent for a cost function $L(\mathbf{W}):=E(V)$. Then for suitable initial conditions, the Euclidean gradient flow, $\dot{W}_i = -\nabla_{W_i} L(W)$, $1\leq i \leq N$
corresponds to the Riemannian gradient flow, $\dot{V} = -\grad_g E(V)$, where the metric $g$ acts by 
\begin{equation}
\label{eq:dln2}
g(Z_1,Z_2)=\mathrm{Tr}(A_{N}^{-1}(Z_1)Z_2), \quad A_N(Z)=\frac1N\sum_{i=1}^N(WW^T)^{\frac{N-i}N}Z(W^TW)^{\frac iN}.
\end{equation}

In order to explore the nature of the Riemannian Langevin equation in optimization, we must understand Brownian motion on Riemannian manifolds like those above. This is a problem of some depth. We illustrate this by computing explicit expressions for Brownian motion in some fundamental cones, using expressions for the barrier from~\cite{Guler}.

\section{Brownian motion and conic programs}\label{sec:rsgd}
Manifold-valued Brownian motion may be defined in several ways~\cite{Hsu}. We use the following definition in this note: an $\mathcal M$-valued semimartingale $\bd X_t$ is called a Brownian motion on $\mathcal M$, with temperature $T=\frac1\beta$, if for any $f\in C^{\infty}(\mathcal M)$, 
\begin{align}
    f(\bd X_t)=f(\bd X_0)+\frac1{\beta}\int_0^t\Delta f(\bd X_s)\ud s+ \mathrm{a\;local\; martingale}.
\end{align}
We denote Brownian motion on $(\mathcal M,g)$ at temperature $T=\frac1\beta$ by $\bd B^{g,\beta}_t$.

\begin{proposition}
\label{prop-qv}
The quadratic variation process of $f(\bd B^{g,\beta}_t)$ for $f\in C^{\infty}(\mathcal M)$ is 
\begin{align}
    [f(\bd B^{g,\beta})]_t=\frac1\beta\int_0^t|\nabla f(\bd B^{g,\beta}_s)|^2_g \ud s
\end{align}
where $|\nabla f|^2:= g^{ij}\partial_if\partial_jf$.
\end{proposition}
\begin{corollary}
The covariation process of $f_1(\bd B^{g,\beta}_t)$ and $f_2(\bd B^{g,\beta}_t)$ for $f_1,f_2\in C^{\infty}(\mathcal M)$ is 
\begin{align}
    [f_1(\bd B^{g,\beta}),f_2(\bd B^{g,\beta})]_t=\frac1\beta\int_0^t (\nabla f_1\cdot\nabla f_2)(\bd B^{g,\beta}_s)\ud s
\end{align}
where $\nabla f_1\cdot\nabla f_2:=g^{ij}\partial_if_1\partial_jf_2$.
\end{corollary}
Proposition~\ref{prop-qv} allows us to analyze Brownian motion through a careful choice of coordinate functions. We will choose $f$ such that $\Delta f=0$ and $|\nabla f|^2=1$, so that $f(\bd B^{g,\beta}_t)-f(\bd B^{g,\beta}_0)$ has the same law as an $\mathbb R$-valued standard Brownian motion. 

Let us now assume $\mathcal{M}$ is a regular convex cone $K \subset \mathbb{R}^n$, let $F$ denote its canonical barrier, and equip $K$ with the Hessian metric 
    \begin{align}
    \label{eq:hessian}
        g_{ij}=\frac{\partial^2F}{\partial x^i\partial x^j}.
    \end{align}
\begin{theorem}\label{prop3} 
Consider Brownian motion $\bd B^{g,\beta}_t$ on $(K,g)$. The process $\frac{\sqrt\beta}{\sqrt n}(F(\bd B^{g,\beta}_t)-F(\bd B^{g,\beta}_0))$ has the same law as a standard Brownian motion on the line.
\end{theorem}
\begin{proof}
We will use the logarithmic homogeneity of $F$ and the Monge-Amp\`{e}re equation to establish the identities
\begin{equation}
\label{eq:cone-identity}
\mathrm{grad}F = -x, \quad |\nabla F|^2=n \quad\mathrm{and}\quad \Delta F=0.
\end{equation}
Theorem~\ref{prop3} follows immediately from these identities.

The first identity uses logarithmic homogeneity. Since $F(\lambda x)=F(x)-n\log \lambda$ for $\lambda \in\mathbb R^+$, $x\in K$, we may differentiate with respect to $\lambda$ and set $\lambda=1$ to find
\begin{align}
    x^i\partial_i F(x)=-n.
\end{align}
Next, the differential of $F$ with respect to $x^j$ is 
$\partial_jF(x)=-x^i\partial_i\partial_j F(x)=-g_{ij}x^i$ by equation~\eqref{eq:hessian}. Thus, 
each component of the gradient of $F$ is
\begin{equation}    
(\mathrm{grad}\,F)^k=g^{jk}\partial_j F=-x^k.
\end{equation}
This proves the first identity in equation~\eqref{eq:cone-identity}. It immediately follows that 
\begin{equation}
\label{eq:grad-identity}
|\nabla F|^2=(g^{ij}\partial_i F)\partial_j F=-x^j\partial_jF=n.
\end{equation}
Finally, we show that $\Delta F=0$ as follows
\begin{align}
    \Delta F=\mathrm{div}(\mathrm{grad}F)&=-\frac{1}{\sqrt{\det g}}\partial_i(\sqrt{\det g}x^i)\nonumber\\
    &=-\partial_i(x^i)-x^i\partial_i(\frac12 \log{\det g})\nonumber =-n-x^i\partial_iF=0,
\end{align}
where we have used the Monge-Amp\'{e}re equation~\eqref{eq:monge-ampere} and equation~\eqref{eq:grad-identity}.
\end{proof}
The above theorem sheds new light on the mysterious reappearance of the Cheng-Yau metric in optimization theory. Let us understand it better with examples.

\section{Brownian motion examples}
\subsection{Positive orthant}
Denoted by $\mathbb R^n_+=\{x\in\mathbb R^n|x^i>0,i=1,...,n\}$ the positive orthant. The canonical barrier and its Hessian metric are
\begin{align}
    F(x)=-\sum_{i=1}^n\log x^i,  \quad   g=\sum_{i=1}^n \frac{\ud x^i\ud x^i}{(x^i)^2}.
\end{align}
Then for each choice of coordinate, we have the identity in law 
\begin{align}
    \log x^i(\bd B_t^{g,\beta})-\log x^i(\bd B_0^{g,\beta})={\frac1{\sqrt\beta}B^i_t}\qquad i=1,...,n
\end{align}
where $\{B^i_t\}_{i=1}^n$ are $n$ independent standard Brownian motion on $\mathbb R$.

\subsection{Cube}
Next we consider a convex {\em set\/}, the cube $B_n=(0,1)^n$. We find that
\begin{align}
    F(x)=-\sum_{i=1}^n \log\frac{\sin(\pi x^i)}{\pi}, \quad    g=\sum_{i=1}^n \pi^2\frac{\ud x^i\ud x^i}{\sin^2(\pi x^i)}.
\end{align}
Similar calculations yield  the identity in law 
\begin{align}
    \log(\tan(\frac{\pi x^i(\bd B^{g,\beta}_t)}2))-\log(\tan(\frac{\pi x^i(\bd B^{g,\beta}_0)}2))=\frac1{\sqrt{\beta}}B^i_t
\end{align}
where $\{B^i_t\}_{i=1}^n$ are $n$ independent standard Brownian motions on $\mathbb R$.

\subsection{Lorentz cone}
A deeper example is provided by the Lorentz cone
\begin{equation}
\label{eq:lorentz}    
K_{n+1}=\left\{x\in\mathbb R^{n+1}|(x^0)\geq \sqrt{\sum_{i=1}^n (x^i)^2}\right\},
\end{equation}
where $x=(x^0, \ldots,x^n)$. The canonical barrier $F$ on $K_{n+1}$ is given by
\begin{align}
\nonumber
    F(x)    &=-\frac{n+1}2\log(x^TAx)+\frac{n+1}2\log(n+1), \quad A=\mathrm{diag}(1,-1,-1,\cdots,-1).
\end{align}
The metric $g$, its inverse, and volume form are as follows 
\begin{align}
    g_{ij}&=-\frac{n+1}{x^TAx}\big(A_{ij}-2\frac{A_{ik}x^kA_{jl}x^l}{x^TAx}\big),\\
    g^{ij}&=-\frac{1}{n+1}\left((x^TAx)B^{ij}-2x^ix^j\right),\\
    \sqrt{\det(g_{ij})}&=\left(\frac{n+1}{x^TAx}\right)^{\frac{n+1}2}=\exp(F),
\end{align}
where $B=A^{-1}$ is the inverse of $A$. In our case we have $B=A$, but these matrices are conceptually distinct. 

We characterize Brownian motion on $K_{n+1}$ using the auxiliary functions
\begin{align}
    f_b(x)=\frac{\sqrt{n+1}}2\log\frac{(b^Tx)^2}{x^TAx}, \quad b \in L_{n+1}^+.
\end{align}
Here we have introduced the light-cone 
\[ L_{n+1}^+=
\left\{b\in\mathbb R^{n+1}|b^0>0, b^TBb=0\right\}. \]
Pick $n$ vectors $\{b_i\}_{i=1}^n\subset L_{n+1}^+$ and define the $n+1$ functions 
 \[f^0=\frac{1}{\sqrt{n+1}}F, \quad\mathrm{and}\quad f^i=f_{b_i}\quad i=1,\ldots,n.\]
We choose a drift and covariance tensor as follows  
\begin{equation}
\label{eq:lorentz-drift}    
\mu^0=0, \quad \mu^i= \frac{n-1}{2\sqrt{n+1}}, \quad i =1, \ldots, n.
\end{equation}
\begin{equation}
\label{eq:lorentz-covariance}    
\Sigma^{ij} = 1-b_i^TBb_j\exp\left(-{\frac{f^i_t+f^j_t}{\sqrt{n+1}}}\right), \quad i,j = 1, \ldots n.
\end{equation}
Finally, set $\Sigma^{00}=1$ and $\Sigma^{ij}=0$ when exactly one of the indices is zero.

\begin{theorem}
Denote by $\bd B_t$ Brownian motion on $(K_{n+1},g)$ with $\beta=2$. The stochastic processes $f^i_t:=f^i(\bd B_t)$ satisfy the \Ito\/ SDE
    \begin{align}
        \ud f^i_t&=\mu^i\ud t+\sigma^i_j\ud B^j_t,\quad f^i_0=f^i(\bd B_0)\qquad i=0,1,...,n.
    \end{align}
In particular, each $f^i_t$ is itself identical in law with a Brownian motion with constant drift. 
\end{theorem}

\begin{proof}
We only need to check that 
    \begin{equation}
    \label{eq:lc1}
        \frac12\Delta f^i=\mu^i, \quad \mathrm{and}\quad \nabla\,f^i\cdot \nabla f^j =\Sigma^{ij}, \quad i,j=0,\ldots,n.
    \end{equation}
First, when $i=j=0$, this is just the claim of Theorem~\ref{prop3}. When exactly one of the indices is zero, we use equation~\eqref{eq:cone-identity} and the fact that $\frac{(b^Tx)^2}{x^TAx}$ is a homogeneous polynomial of order $0$ to obtain 
    \begin{align}
        \nabla f^0\cdot \nabla f^i=\frac{1}{\sqrt{n+1}}\mathrm{grad }\, F(\ud f^i)=-\frac{1}{\sqrt{n+1}}x^k\partial_kf^i=0.
    \end{align}

Finally, consider the case when both $i$ and $j$ are space-like. We start with the following property: for $b,b'\in L_{n+1}^+$,
    \begin{align}
        \nabla \log (b^Tx)\cdot\nabla \log(b'^Tx) &=-\frac{1}{n+1}\left((x^TAx)B^{ij}-2x^ix^j\right)\frac{b'_i}{b^Tx}\frac{b_j}{b^Tx}\nonumber\\
        &=\frac{1}{n+1}\left( 2-(b^TBb')\frac{(x^TAx)}{(b^Tx)(b'^Tx)}\right).
    \end{align}

    Using the fact that $\log(b^Tx)=\frac{1}{\sqrt{n+1}}(f_b-f^0)$ and $\nabla f_b\cdot \nabla f^0=0$, we have  
    \begin{align}
        \nabla f_b\cdot\nabla f_{b'}&=(n+1)\nabla \log (b^Tx)\cdot\nabla \log(b'^Tx)-|\nabla f^0|^2\nonumber\\
        &=1-(b^TBb')\frac{(x^TAx)}{(b^Tx)(b'^Tx)}\nonumber\\
        &=1-(b^TBb')\exp\left(-\frac{f_b+f_{b'}}{\sqrt{n+1}}\right).
    \end{align}
In particular, we find that $|\nabla f_b|^2=1$ because $b^TBb=0$ when $b \in L_{n+1}^+$. 

The proof of the first identity in equation~\eqref{eq:lc1} is a computation:
    \begin{align}
        \big(\mathrm{grad}\,\log(b^Tx)\big)^i&=-\frac{1}{n+1}((x^TAx)B^{ij}-2x^ix^j)\frac{b_i}{b^Tx}=\frac{1}{n+1}(2x^i-\frac{x^TAx}{b^Tx}B^{ij}b_j)\nonumber\\
        \Rightarrow \Delta \log(b^Tx)&=\frac{1}{\sqrt{\det g}}\partial_i\big(\sqrt{\det g}(\mathrm{grad }\,\log(b^Tx))^i\big)\nonumber\\
        &=\partial_i\big((\mathrm{grad }\,\log(b^Tx))^i\big)+\big(\mathrm{grad }\,\log(b^Tx)\big)^i\partial_i F\nonumber\\
        &=\partial_i\big(\frac1{n+1}(2x^i-\frac{x^TAx}{b^Tx}B^{ij}b_j)\big)+(\mathrm{grad }\,F)^i\partial_i \log(b^Tx)\nonumber\\
            &=\frac1{n+1}(2(n+1)-2)-1=\frac{n-1}{n+1}.
    \end{align}
Thus, finally we have   
\[ \Delta f_b=\Delta(\sqrt{n+1}\log(b^Tx)+f^0)=\frac{n-1}{\sqrt{n+1}}.\]

\end{proof}
\section{Acknowledgements}
This work was supported by NSF grant DMS-2107205.

%
%

%
%
%
\bibliographystyle{splncs04}
\bibliography{bib}

\end{document}